\documentclass[a4paper,reqno, 12pt]{amsart}

\title{Rational lines on cubic hypersurfaces}

\author{Julia Brandes}
\address{JB: Mathematical Sciences, Chalmers Institute of Technology and University of Gothenburg, 412 96 G{\"o}teborg, Sweden}
\email{brjulia@chalmers.se}

\author{Rainer Dietmann}
\address{RD: Department of Mathematics, Royal Holloway, University of London, Egham,
TW20 0EX, UK}
\email{rainer.dietmann@rhul.ac.uk}


\usepackage[english]{babel}
\usepackage[T1]{fontenc}
\usepackage[latin1]{inputenc}
\usepackage{latexsym}
\usepackage{epsfig}
\usepackage{bm}
\usepackage{amssymb}
\usepackage{amsmath}
\usepackage{verbatim}
\usepackage{amsthm}
\usepackage{mathabx}
\usepackage{enumerate}
\usepackage{mathrsfs}
\usepackage[hmargin=3cm,vmargin=3cm]{geometry}

\makeatletter
\def\Ddots{\mathinner{\mkern1mu\raise\p@
\vbox{\kern7\p@\hbox{.}}\mkern2mu
\raise4\p@\hbox{.}\mkern2mu\raise7\p@\hbox{.}\mkern1mu}}
\makeatother

\def\A{\mathbb A}
\def\C{\mathbb C}
\def\F{\mathbb F}

\def\N{\mathbb N}
\def\PP{\mathbb P}
\def\Q{\mathbb Q}
\def\R{\mathbb R}
\def\SS{\mathbb S}

\def\Z{\mathbb Z}

\def\L{\mathbb L}

\def\B#1{\mathbf{#1}}

\def\cal#1{\mathcal{#1}}

\def\mmod#1{\;(\mathrm{mod}\;{#1})}

\DeclareMathOperator{\rk}{rank}

\DeclareMathOperator{\sing}{Sing}

\DeclareMathOperator{\Span}{Span}
\DeclareMathOperator{\disc}{Disc}

\renewcommand\le{\leqslant}
\renewcommand\ge{\geqslant}

\newtheorem{thm}{Theorem}[section]
\newtheorem{lem}[thm]{Lemma}
\newtheorem{cor}[thm]{Corollary}

\theoremstyle{definition}

\theoremstyle{remark}

\numberwithin{equation}{section}

\newenvironment{pf}{\begin{proof}[Proof]}{\end{proof}}

\begin{document}

\begin{abstract}
    We show that any smooth projective cubic hypersurface of dimension at least $29$ over the rationals contains a rational line. A variation of our methods provides a similar result over $p$-adic fields. In both cases, we improve on previous results due to the second author and Wooley.

    We include an appendix in which we highlight some slight modifications to a recent result of Papanikolopoulos and Siksek. It follows that the set of rational points on smooth projective cubic hypersurfaces of dimension at least 29 is generated via secant and tangent constructions from just a single point.
\end{abstract}

\maketitle

\section{Introduction}

A celebrated theorem of Birch \cite{Bir:57} asserts that any rationally defined projective hypersurface of odd degree contains large rational linear spaces, provided only that the dimension of the hypersurface is sufficiently large in terms of its degree and the dimension of the linear space. Unfortunately, due to the great generality of the result, the bounds on the number of variables required are, by a felicitous expression of Birch, not even astronomical (see Wooley \cite{W:97fmv, W:98bir} for explicit results). By focussing on relevant special cases, however, one would expect to be able to take advantage of a richer box of tools tailored for purpose, and thereby to establish much improved bounds compared to the general case. For instance, Wooley \cite[Theorem 2(b)]{W:97linspcub} proved that every cubic hypersurface defined by a rational homogeneous polynomial in at least $37$ variables contains a rational line. Our first result in the work at hand improves this bound in the non-singular case.
\begin{thm}\label{t:lines}
    Let $F \in \Z[x_1, \ldots, x_n]$ be a non-singular cubic form in $n \ge 31$ variables. Then the hypersurface $\cal V$ defined by $F(\B x)=0$ contains a rational line.
\end{thm}\smallskip

Observe that, unlike in Wooley's work, we require a smoothness condition, although this may be somewhat relaxed at the cost of inflating the number of variables correspondingly (see e.g. the argument in Section 4 of \cite{B:15p}). Just like in the earlier work, we approach the problem by first fixing one rational point on the hypersurface, and then constructing a second point with the property that the line defined by these two points is contained in the hypersurface. As we will see, this strategy reduces to finding a rational solution to a system of one linear, one quadratic and one cubic equation. Whilst Wooley approaches this problem by identifying a large linear space on the quadratic hypersurface on which to solve the cubic, we will instead employ a recent result due to the second author jointly with Browning and Heath-Brown \cite{BDHB:15}. These authors establish a local-global principle for systems of one cubic and one quadratic equation under the condition that the underlying variety be smooth of sufficiently high dimension. However, in order to be able to make use of such a theorem, we need to ensure that the system has non-trivial solutions over $\R$, so we need to arrange for the quadratic equation arising in the system to satisfy some indefiniteness condition. This is in fact the most delicate part of the argument.\smallskip

As a corollary of Theorem \ref{t:lines}, we obtain a new result regarding prime solutions of polynomials. There has been some interest recently in trying to determine under what conditions general homogeneous polynomials have solutions in the primes (see \cite{LiuJY}, \cite{CM}). When the polynomials are diagonal, there is a vast amount of literature and our understanding of this problem approaches that of the integer case (see in particular Hua's book \cite{Hua:ATPN}, and \cite{KW:17} for recent results). For general polynomials of degree greater than two the bounds are much less satisfactory. However, if we weaken the condition that the variables should be prime, Theorem~\ref{t:lines} can be marshalled to improve upon the bounds in the cubic case.
\begin{cor}
    Let $F \in \Z[x_1, \dots, x_n]$ be a non-singular cubic form in at least $n \ge 31$ variables. Then there are integers $c_1, \ldots, c_n$ with $\gcd(c_1, \dots, c_n)=1$ such that the equation
    \begin{align*}
        F(c_1 p_1, c_2 p_2, \ldots, c_n p_n)=0
    \end{align*}
    has infinitely many solutions with $p_1, \ldots, p_n$ all prime and not all equal.
\end{cor}
This follows upon inserting Theorem~\ref{t:lines} into Lemma~3 of Br\"udern, Dietmann, Liu and Wooley \cite{BDLW}, and it improves upon the previous bound $n \ge 37$ obtained in the same manner from \cite[Theorem 2(b)]{W:97linspcub}.\smallskip

The keen reader will be poised to inquire about lower bounds on the number of variables necessary to guarantee the existence of rational lines on cubic hypersurfaces. A necessary condition for a hypersurface to contain a rational line is that it should contain a line over all completions of $\Q$. In this vein, the second author and Wooley \cite[Theorem~1]{DW:03} showed that, for all primes $p$, every cubic hypersurface in at least $23$ variables contains a $p$-adic line. We are able to reduce this bound further.
\begin{thm}\label{t:Qp}
	Let $p$ be a  rational prime and let $F \in \Q_p[x_1, \ldots, x_n]$ be a cubic form defining a hypersurface $\cal V=
\{\mathbf{x} \in \Q_p^n: F(\mathbf{x})=0\}$. Then $\cal V$
contains a line defined over $\Q_p$ if at least
one of the two following conditions is satisfied:
	\begin{enumerate}[(a)]
		\item \label{it:allp}
		$n \ge 22$, or
                \item[or]
		\item \label{it:largep}
		$n \ge 16$ and $p \ge 19$. 
	\end{enumerate}			
\end{thm}
On the other hand, we are able to construct cubic forms in $11$ variables which define a variety devoid of $\Q_p$-lines for certain small primes. 
\begin{thm}\label{t:lower}
    For every $p \in \{2, 3, 5\}$ there exists an explicit smooth cubic hypersurface defined by a non-singular homogeneous integer cubic polynomial in $11$ variables that fails to contain a line over $\Q_p$.
\end{thm}
The authors are not aware of any viable heuristic argument indicating the true bound on the number of variables required to guarantee the existence of a $p$-adic or rational line. Unlike in the complex case, where it is well known that every smooth cubic surface contains 27 lines, the situation over $\Q$ and $\Q_p$ is much less satisfactorily understood. In particular, given that at least ten variables are needed to ensure that a cubic hypersurface contains a $p$-adic point for all primes $p$, it seems unlikely that the result in Theorem~\ref{t:lower} should be best possible. On the other hand, the proofs of Theorems~\ref{t:lines} and \ref{t:Qp} admit for a substantial amount of flexibility in the construction of the line. It thus seems conceivable that our proofs may be strengthened to show that, within a suitable bounded region on the hypersurface, almost every point lies on a line. Consequently, one might speculate that the true bound on the number of variables required to ensure the existence of a rational or $p$-adic line on $\cal V$, without any information regarding the density of such lines, could be substantially lower and might even lie in the proximity of the lower bound of Theorem~\ref{t:lower}. 
\smallskip

We include an appendix on recent work by Papanikolopoulos and Siksek \cite{PapaSik}. These authors used the results of \cite{BDHB:15} in order to show that for any smooth projective rationally defined cubic surface $\cal V$ of dimension at least $48$, the set $\cal V(\Q)$ of rational points on $\cal V$ can be obtained via secant and tangent operations from just a single point. We sketch a slight modification to their argument that shows how the same conclusion can be obtained for all smooth rationally defined cubic surfaces $\cal V$ of dimension at least $29$. \medskip

\textbf{Acknowledgements.} This work was conceived during a visit of the second author at the Department of Mathematics at Chalmers/University of Gothenburg, followed by a visit of both authors to the Fields Institute in Toronto during the Thematic Program on Unlikely Intersections, Heights, and Efficient Congruencing. The authors would like to thank both institutions for excellent working conditions. Likewise, we are indebted to Karim Johannes Becher, Johan Bj\"orklund, Tim Browning, Oscar Marmon and Per Salberger for many enlightening discussions. The first author was supported by the National Science Foundation under Grant No. DMS-1440140 while being in residence at the Mathematical Sciences Research Institute in Berkeley, California, during the Spring 2017 semester, as well as a Starting Grant from the Swedish Research Council under grant agreement no. 2017-05110.

\section{The construction of rational lines}\label{s:line}

We will assume throughout that $F \in \Z[x_1, \dots, x_n]$ is a non-singular cubic form defining a hypersurface $\cal V \subseteq \A^n$. We then define the polar form $\Phi$ associated with $F$ as the unique symmetric trilinear form with the property that $\Phi(\B x, \B x, \B x) = F(\B x)$; this form has coefficients in $\Z/6$. In this notation, two points $\B x, \B y \in \A^n$ define a line on $\cal V$ if they are linearly independent and satisfy 
\begin{align*}
	F(\B x) = \Phi(\B x, \B x, \B y) = \Phi(\B x, \B y, \B y) = F(\B y) = 0.
\end{align*}
We therefore aim to find a base point $\B x_0 \in \cal V(\Q)$ and then solve the residual system of equations in $\B y$. This system consists of one cubic, one quadratic and one linear equation, with the coefficients of the latter two depending on the base point $\B x_0$. In particular, $\B x_0$ should be chosen in such a way that the quadratic equation is indefinite and such as to provide the greatest possible control over the singularities of the system. To this end, we need to establish some facts about the real geometry of $\cal V$. 

It will be convenient to work on the unit hypersphere $\SS^{n-1} = (\R^n \setminus \{\bm 0\} )/\!\sim$, where two points $\B x$, $\B x'$ are considered equivalent if $\B x = \lambda \B x'$ for some $\lambda \in \R_+$. In a similar vein, we write $\SS^{n-1}(\Q) = (\Q^n \setminus \{\bm 0\} )/\!\!\sim$, where we note that the scaling factor $\lambda$ immanent in the equivalence relation still runs over all positive real numbers. We can map $\SS^{n-1}$ onto the real projective space $\R\PP^{n-1}$ by identifying the points $\B x$ and $-\B x$; this map is evidently a double cover of $\R\PP^{n-1}$. In analogy to the usual notational conventions, whenever $X \subseteq \A^n$ is a projective variety, we write $X(\SS)$ when $X$ is viewed as a variety over the unit hypersphere. In a similar manner, we use the notation $\dim_\SS X$ to denote the dimension of $X$, viewed as a manifold on $\SS^{n-1}$. The definition of $\SS^{n-1}$ here implies that $\dim_\SS X = \dim_{\R\PP} X = \dim_\R X -1$. 

Whilst the following statement can be derived from standard facts on the projective geometry of real cubic hypersurfaces such as presented in \cite[Section 4.3]{viro}, we choose to give a fairly elementary, self-contained proof of the result in the formulation relevant to our setting. 

\begin{lem}\label{l:cc}
    Let $F \in \R[x_1, \ldots, x_n]$ be a non-singular homogeneous real cubic polynomial defining a variety $\cal V$, and suppose that $n \ge 3$. Then $\cal V(\SS)$ has a connected component $U \subseteq \SS^{n-1}$ with dimension $n-2$ over $\SS$, having the property that $U = - U$.
\end{lem}
\begin{pf}
    Consider the affine variety $\cal V(\R)\subseteq \R^n$. After a suitable linear change of variables, we may assume that $F$ is cubic in $x_n$. Given any point $\B y \in \R^{n-1}$, set $f_{\B y}(x) = F(\B y, x)$. The polynomials $f_{\B y}$ are real cubic polynomials in one variable, and thus have either one or three real zeros (with multiplicity). The coefficients of the polynomials $f_{\B y}$ are continuous functions of $\B y$. It therefore follows that the zeros of the polynomials $f_{\B y}$ vary continuously with $\B y$, and they lie in a bounded domain whenever $\B y$ lies in a bounded domain.

    Fix a generic point $\B y_0 \in \R^{n-1}$, and pick a path $\ell \subseteq \R^{n-1}$ from $\B y_0$ to $- \B y_0$ avoiding the origin, so that $\ell(0)=\B y_0$ and $\ell(1) = -\B y_0$. For $-1 \le t \le 0$ set $\ell(t) = -\ell(1+t)$, then $\ell: [-1,1] \to \R^{n-1}$ describes a closed symmetric loop through $\B y_0$ and $- \B y_0$. It is clearly possible to choose $\ell$ generically in such a way that this loop is homeomorphic to $\SS^1$. Furthermore, since $F$ is homogeneous, it is clear that whenever $x$ is a zero of $f_{\B y}$ for some $\B y \in \R^{n-1}$, then $-x$ is a zero of $f_{-\B y}$. 
    
    Let $C$ denote the locus of points $(\ell(t), x(t))$ satisfying $f_{\ell(t)}(x(t))=0$, as $t$ ranges over $[-1,1]$. Here, $x(t)$ denotes the locus of solutions to $f_{\ell (t)}$ and may be viewed as a multi-valued map. Since the solutions $x(t)$ vary continuously with $t$, the set $C$ consists of a number of (possibly intersecting) closed loops on the cylinder $\ell \times \R$. Denote by $C^*$ the union of the incontractible loops in $C$. There is a homotopy transformation $\phi$ of $\ell \times \R$ with the property that $C'=\phi(C^*)$ has the same multiplicity over $\ell(t)$ for all points $t \in [-1,1]$; this transformation has the effect of straightening out any S-shaped bends within $C^*$. Since every contractible loop and every S-bend contribute an even number of additional zeros to the polynomials $f_{\ell(t)}$, and $C$ covers each point of $\ell$ with multiplicity either one or three, the same remains true for $C'$. Note also that, since $C$ and thus $C^*$ is symmetric with respect to the origin, the transformations underlying the homotopy between $C^*$ and $C'$ can be assumed to be symmetric with respect to the origin. Hence whenever there is a path $\gamma$ connecting two antipodal points on $C'$, the pre-image $\phi^{-1}(\gamma)$ of that path connects two antipodal points on $C^*$ and therefore also on $C$.

    Consider first the case where $C'$ covers $\ell$ with multiplicity $1$. In this case, $C'$ consists of one loop connecting a zero of $f_{\B y_0}$ with a zero of $f_{- \B y_0}$. Since $C'$ is symmetric with respect to the origin, and the map $\phi$ had been symmetric, this loop connects two antipodal points.

    Now consider the case that $C'$ covers $\ell$ with multiplicity $3$. Since the operations performed in order to simplify $C$ to $C'$ never increase the multiplicity by which the set passes over a certain point $\ell(t)$, it follows that $C$ itself covers $\ell$ with multiplicity $3$. This is possible either as one loop with winding number $3$, or three loops with winding number $1$ each. The case of there being two loops with winding numbers $1$ and $2$, respectively, is excluded by symmetry, as any loop with winding number $2$ would have an odd number of self-intersection points on the cylinder $\ell \times \R$, an impossible feat for a graph that is symmetric with respect to the origin. In the case of there being one loop with winding number $3$, this loop runs through all three zeros of $f_{\B y_0}$ and also through all three zeros of $f_{-\B y_0}$, thus forming a path between antipodal points. If there are three loops with winding number $1$ each, we use symmetry again. Let $y_1, y_2, y_3$ be the zeros of $f_{\B y_0}$, then by symmetry the zeros of $f_{- \B y_0}$ are given by $-y_1, -y_2, -y_3$. If $(\B y_0, y_1)$ and $(-\B y_0, -y_1)$ lie on the same loop, the statement is established, so assume without loss of generality that $(\B y_0, y_1)$ lies on the same loop as $(-\B y_0,-y_2)$. In this case symmetry demands that $(\B y_0, y_2)$ and $(-\B y_0,-y_1)$ are joined by a common loop, and we are forced to 
    conclude that the third loop runs through $(\B y_0,y_3)$ and $(-\B y_0,-y_3)$ and thus connects two 
    antipodal points. 
    
    We have therefore shown that there exists a path $\gamma \subseteq \cal V$ connecting a pair of antipodal points $\B x_0$ and $- \B x_0$ on $\cal V$. Since the path $\gamma$ avoids the origin, it follows that both points $\B x_0$ and $- \B x_0$ lie on the same connected component $U$ of $\cal V(\SS)$. Furthermore, when $\B x_1 \in U$ is arbitrary, then it lies on the same connected component as $\B x_0$. By symmetry, $-\B x_1$ lies on the same connected component as $-\B x_0$ and thus on $U$ also. We infer that $U = -U$. Finally, since $F$ is non-singular, the point $\B x_0 \in \cal V$ is non-singular over $\R$, which implies that $U$ is of the expected dimension over $\SS$. 
\end{pf}

We now return to the system 
\begin{align}\label{eq:sys-1}
	\Phi(\B x_0, \B x_0, \B y) = \Phi(\B x_0, \B y, \B y) =F(\B y)= 0.
\end{align}
Denote by $H_{\B x}$ the matrix associated to the quadratic form $\Phi(\B x, \B y, \B y)$. This matrix is up to a constant factor equal to the Hessian of $F$ at $\B x$, and we will use both terms interchangeably. We then set
\begin{align*}
    \cal H(k)= \{\B x \in \A^{n}: \rk H_{\B x} \le k\}.
\end{align*}
For a square matrix $M$ we write $\sigma^+(M)$ and $\sigma^-(M)$ for the number of positive and negative eigenvalues of $M$, respectively, and set $\sigma(M) = \min\{\sigma^+(M), \sigma^-(M)\}$.
We can now prove the main auxiliary result that is needed for the proof of Theorem~\ref{t:lines}.
\begin{lem}\label{l:realmainlem}
    Let $F$ be a non-singular real cubic form with polar form $\Phi$, and let $\cal V$ denote the hypersurface defined by the vanishing of $F$. Furthermore, assume that $n \ge 4$. There exists $\B x_* \in \SS^{n-1}$ such that
    \begin{enumerate}[(i)]
        \item $F(\B x_*)=0$, and
        \item\label{it:2} we have $\rk H_{\B x_*} \ge n-1$, and
        \item\label{it:3} the signature of $H_{\B x_*}$ satisfies
             \begin{align}\label{eq:bound-signH}
                \sigma(H_{\B x_*}) \ge \lfloor \textstyle{\frac12}(n-2)\rfloor.
             \end{align}
    \end{enumerate}
\end{lem}

\begin{proof}
    For the proof of this lemma we will work exclusively over $\SS^{n-1}$; observe that this is possible since all sets we will be considering are invariant under scaling.

    It follows from Lemma~\ref{l:cc} that the variety $\cal V(\SS)$ has a connected component $U$ of the expected dimension with the property that $U=-U$. As the dimension of a manifold over $\R$ is always bounded above by its dimension over $\C$, and since $F$ is non-singular,
it follows further from \cite[Lemma~2]{HB:83} that
    \begin{align}\label{dimH}
        \dim_{\SS} \cal H(n-k) \le \dim_\C \cal H(n-k)-1 \le  n-k-1 \qquad (1 \le k \le n-1).
    \end{align}
    Since $\dim_{\SS}U = n-2$, we infer that $\rk H_{\B x} \ge n-1$ for generic $\B x \in U$. If the matrix $H_{\B x}$ already satisfies \eqref{eq:bound-signH} for such a point $\B x$, we may take $\B x_* = \B x$. Otherwise, suppose without loss of generality that the eigenvalues of $H_{\B x}$ are predominantly positive, so that we have $\sigma^{-}(H_{\B x}) < \lfloor \textstyle{\frac12}(n-2) \rfloor$, the other case being analogous. Since $\B x \in U$ and $U$ is symmetric, we know that $-\B x \in U$ as well and we can find a path $\gamma \subseteq U$ connecting the two. The eigenvalues of $H_{\B z}$ vary continuously as $\B z$ moves along $\gamma$, so we need to be able to pick our path $\gamma$ such that the rank of $H_{\B z}$ is bounded below for all $\B z \in\gamma$.
    
    From \eqref{dimH} we discern as before that $\dim_{\SS} U - \dim_{\SS} \cal H(n-3) \ge 2$. As $F$ is non-singular,
we are therefore at liberty to choose a path that avoids the set $\cal H(n-3)$, so that for all $\B z \in \gamma$ no more than two of the eigenvalues of $H_{\B z}$ vanish simultaneously. Since $F$ is of odd degree, we have $\sigma^+(H_{\B x})= \sigma^-(H_{-\B x})$. Hence by travelling along $\gamma$, we move from a point with a predominantly positive signature to a point where the eigenvalues of the Hessian are predominantly negative. Since the interval $\left[\lfloor \textstyle{\frac12}(n-2) \rfloor, n-\lfloor \textstyle{\frac12}(n-2) \rfloor\right]$ always contains at least three integers, it follows that there exists an open set $\gamma' \subseteq \gamma$ with the property that all points $\B x \in \gamma'$ satisfy \eqref{eq:bound-signH}.

    It remains to confirm condition \eqref{it:2}. Whilst without further information we cannot force the intersection $\gamma \cap \cal H(n-2)$ to be empty, it follows from B\'ezout's theorem \cite[Theorem 18.4]{Har:AG} that for generic algebraically defined $\gamma$ it contains at most finitely many points. Hence the set $\gamma' \setminus \cal H(n-2)$ is non-empty, and we conclude that there is a point $\B x_* \in \gamma' \setminus \cal H(n-2)$ which by construction satisfies all the properties postulated in the lemma.
\end{proof}

We can now complete the proof of Theorem~\ref{t:lines}.
Consider the point $\B x_*$ constructed in Lemma~\ref{l:realmainlem}, and set
\begin{align*}
    \cal K = \left\{\B z \in \A^n: \sigma(H_{\B z}) < \lfloor \textstyle{\frac12}(n-2)  \rfloor\right\} \cup \cal H(n-2)
\end{align*}
for the set of points violating condition \eqref{it:2} or \eqref{it:3} of Lemma \ref{l:realmainlem}. From the definition of $\sigma(H_{\B z})$ and the fact that the eigenvalues of $H_{\B z}$ vary continuously with $\B z$, we see that the set $\left\{\B z \in \A^n: \sigma(H_{\B z}) \ge \lfloor \textstyle{\frac12}(n-2)  \rfloor\right\}$ is open in standard topology. Consequently, $\cal K$ is a closed set which is clearly invariant under scaling, and can therefore be interpreted as a submanifold in $\SS^{n-1}$.

 On $\SS^{n-1}$ we have the great-circle distance $d_{\SS}(\B x, \B y)$, which for non-antipodal points $\B x, \B y$ is given by the length of the shorter arc of the geodesic passing through $\B x$ and $\B y$. The natural extension to antipodal points is then given by $d_{\SS}(\B x, -\B x) = \pi$ for all $\B x \in \SS^{n-1}$.
Thus, we may define
\begin{align*}
    \tau =  \min_{\B y \in \cal K(\SS)} d_{\SS}(\B x_*, \B y).
\end{align*}
Since the unit hypersphere is compact and the set $\cal K$ is a closed set as a subset of $\SS^{n-1}$ in standard topology, this distance is well-defined, and it is strictly positive by construction.

It is a well-known fact that smooth cubic hypersurfaces of dimension greater than or equal to $15$ satisfy weak approximation (see for instance \cite[Corollary 1]{Sk:97}). We can therefore find a rational point $\B x_0 \in \cal V \cap \SS^{n-1}(\Q)$ with $d_{\SS}(\B x_*, \B x_0) < \tau/2$. It then follows from the definition of $\tau$ that $\B x_0$ inherits the properties of Lemma~\ref{l:realmainlem} from $\B x_*$, so in particular we have $\rk H_{\B x_0} \ge n-1$ and $\sigma(H_{\B x_0}) \ge\lfloor \textstyle{\frac12}(n-2) \rfloor$.

It remains to solve the system \eqref{eq:sys-1} for $\B y$. This system, as a system of equations in $\B y$, consists of one linear form $\Phi(\B x_0, \B x_0, \B y)$, one quadratic form $Q_{\B x_0}(\B y) = \B y \cdot H_{\B x_0}\B y$ defining a hypersurface $\cal Q$ with $\dim \sing \cal Q \le 0$, and the non-singular cubic form $F(\B y)$. Since $\B x_0$ is rational, all of these polynomials have rational coefficients. The linear form defines a hyperplane $Y$, and it follows that $\dim \sing (\cal V \cap Y) \le 0$ and $\dim \sing (\cal Q \cap Y) \le 1$. We now intersect $\cal Q \cap \cal V \cap Y$ with an additional generic rationally defined hyperplane $Z$, which is chosen such as to not contain $\B x_0$. By Bertini's theorem, this hyperplane can be chosen in such a way that $\cal V \cap Y \cap Z$ is smooth and $\dim \sing (\cal Q \cap Y \cap Z) \le 0$. After a change of variables, we now see that $F|_{Y \cap Z}$ is a non-singular rational cubic form in the variables $y_1, \ldots, y_{n-2}$, and $Q_{\B x_0}|_{Y \cap Z}$ is a rational quadratic form in the variables $y_1, \ldots, y_{n-2}$ of rank at least $n-3$. We now use the following result due to the second author jointly with Browning and Heath-Brown.

\begin{thm}\label{t:bdhb}
    Let $C \in \Q[x_1, x_2, \ldots, x_\nu]$ be a non-singular cubic form in $\nu \ge 29$ variables, and let $Q \in \Q[x_1, \dots, x_\nu]$ be a quadratic form with $\rk Q \ge 28$ that satisfies
    \begin{align}\label{eq:BDHB-realcond}
        \sigma(Q) \ge 2.
    \end{align}
    Then the system 
    \begin{align*}
        C(\B x) = Q(\B x) = 0
    \end{align*}
    has a non-trivial rational solution.
\end{thm}
\begin{pf}
    This is a rephrasing of Theorem~1.3 of \cite{BDHB:15}, and it follows upon combining Lemmata 7.2 and 8.1 of \cite{BDHB:15}. The only part requiring justification are the local solubility conditions.

    The existence of non-singular $p$-adic solutions of the system $C=Q=0$ follows from \cite[Lemma 9.1]{BDHB:15}, after noting that $\nu \ge 29$ and the non-singularity of $C$ imply the assumptions $\operatorname{ord}_Q(C) \ge 10$ and $h_Q(C) \ge 2$ there. Moreover, by \eqref{eq:BDHB-realcond} the quadratic form $Q$ vanishes on a two-dimensional $\R$-linear space $V$ and is non-singular on $V$. Substituting $V$ into $C$, and using the fact that every real binary cubic form has a non-trivial real zero, we find that $C$ vanishes on a real point $P$ which may be taken to be a non-singular point on $Q=0$.
    One then checks that from this point the proof of \cite[Lemma 9.1]{BDHB:15} goes through over $\R$, producing a non-singular real solution $\B x$ to the system
    $C(\B x)=Q( \B x)=0$.

    The arguments of \cite[Section 8]{BDHB:15} (see also \cite[Section 10]{HB:83}) can now be applied to show that the local factors in Lemma~8.1 of that paper are positive, which establishes the desired conclusion.
\end{pf}

Observe that the conclusion \eqref{eq:bound-signH} of Lemma~\ref{l:realmainlem} ensures that the quadratic form satisfies \eqref{eq:BDHB-realcond}, so the real solubility condition is satisfied and Theorem~\ref{t:lines} follows from Theorem~\ref{t:bdhb} as soon as $n-2 \ge 29$.

\section{The construction of $p$-adic lines}
For the proof of Theorem~\ref{t:Qp}, we will work over suitable quadratic field extensions $\L$ of $\Q_p$. Throughout this section, despite all geometric objects being projective, we prefer to use the term dimension in the affine sense.
We first establish an auxiliary result regarding linear spaces on quadratic hypersurfaces.
\begin{lem}\label{l:quad}
    Let $p$ be a rational prime and $Q \in \Q_p[x_1, \ldots, x_4]$ be a quadratic form. Then either $Q$ vanishes on a two-dimensional linear space over $\Q_p$, or else there exists a quadratic extension $\L$ of $\Q_p$ such that $Q$ vanishes on a two-dimensional $\L$-linear space.
\end{lem}
\begin{proof}
    Let us first discuss the case when $Q$ is singular. Then, after a suitable non-singular linear transformation on the variables, we can assume without loss of generality that $Q$ is of the form
    \begin{align*}
        Q(x_1, \ldots, x_4) = T(x_1, x_2, x_3)
    \end{align*}
    for some ternary quadratic form $T \in \Q_p[x_1, x_2, x_3]$. Let
    \begin{align*}
        B(x_1, x_2) = T(x_1, x_2, 0),
    \end{align*}
    and write $\disc(B)$ for the determinant of the matrix associated with $B$. 
    Set $\L=\Q_p$ if $- \disc(B)$ is a square in $\Q_p$ and $\L=\Q_p(\sqrt{- \disc(B)})$ otherwise. Then $B$ and thus $T$ are isotropic over $\L$, and $Q$ vanishes on the two-dimensional linear space spanned by a zero of $T$ and $(0, 0, 0, 1)$.

    Let us now assume that $Q$ is non-singular. First suppose that $Q$ is isotropic over $\Q_p$. Then it is possible by a suitable non-singular linear transformation of the variables to split off a hyperbolic plane, so that we may assume without loss of generality that $Q$ is of the form
    \begin{align*}
        Q(x_1, \ldots, x_4) = x_1 x_2 + B(x_3, x_4)
    \end{align*}
    for some quadratic form $B \in \Q_p[x_3, x_4]$. As before, set $\L = \Q_p$ if $- \disc(B)$ is a square in $\Q_p$ and $\L=\Q_p(\sqrt{- \disc(B)})$ otherwise. The form $B$ is then isotropic over $\L$, so again $Q$ vanishes on a two-dimensional $\L$-linear space.

    It remains to discuss the case when $Q$ is anisotropic over $\Q_p$. In this setting, it follows from Serre \cite[p.39, Corollary]{S:CArith} that, after a suitable non-singular linear transformation of the variables, we may assume without loss of generality 
    that $Q$ is of the form
    \begin{align*}
        Q(x_1, \ldots, x_4) = x_1^2 - a x_2^2 - b x_3^2+a b x_4^2
    \end{align*}
    for suitable coefficients $a, b \in \Q_p$ having the property that $-ab$ is not a square in $\Q_p$. Let $\L=\Q_p(u)$ where $u^2=-ab$. Then both forms
    \begin{align*}
        x_1^2+a b x_4^2=x_1^2-(u x_4)^2\qquad \text{ and }\qquad -a x_2^2-b x_3^2=-\frac{1}{b}((b x_3)^2-(u x_2)^2)
    \end{align*}
    are isotropic over $\L$. We conclude that again, $Q$ vanishes on a two-dimensional linear space defined over $\L$.
\end{proof}

With the help of Lemma~\ref{l:quad} we can find a large linear space on the hypersurface defined by the quadratic form $Q$.

\begin{lem}\label{l:bigquad}
	Let $t \in \N$ with $t \ge 1$, let $p$ be a rational prime, and let $Q(x_1, \ldots, x_s) \in \Q_p[x_1, \ldots, x_s]$ be a quadratic form with $s \ge 2t$. Then either $Q$ vanishes on a $t$-dimensional linear space over $\Q_p$, or there exists a quadratic extension $\L$ of $\Q_p$ such that $Q$ vanishes on a $t$-dimensional $\L$-linear space.
\end{lem}

\begin{proof}
	By setting $x_{2t+1} = \ldots = x_{s}=0$ in case of $s>2t$, we can reduce to the case $s=2t$, which we assume from now on. We will prove the result by induction on $t$. The conclusion is trivial when $t=1$, and the case $t=2$ is just Lemma \ref{l:quad}. 
	
	Now suppose that the result has already been established for some value $t \ge 2$, and consider a quadratic form in $s = 2t+2$ variables.  	
	Let us first discuss the case that $Q$ is non-singular. Since $2t +2 \ge 5$, the quadratic form $Q$ is isotropic over $\Q_p$. We can thus split off a hyperbolic plane. Then after a suitable non-singular linear change of variables, we can without loss of generality assume that $Q$ is of the form
	\begin{align*}
  		Q(x_1, \ldots, x_{2t+2}) =  Q_1(x_1, \ldots, x_{2t}) +  x_{2t+1} x_{2t+2}
	\end{align*}
	for some quadratic form $Q_1 \in \Q_p[x_{1}, \ldots, x_{2t}]$. By our inductive assumption, $Q_1$ vanishes on a $t$-dimensional linear space over either $\Q_p$ or a suitable quadratic extension $\L$ of $\Q_p$, whence $Q$ vanishes on a $(t+1)$-dimensional linear space over $\Q_p$ or $\L$, respectively. Hence the conclusion holds true in case of non-singular $Q$. 
	
	Now suppose that $Q$ is singular. Then after a suitable non-singular linear change of variables, we can without loss of generality assume that $Q$ is of the form
	\begin{align*}
  		Q(x_1, \ldots, x_{2t+2})=T(x_1, \ldots, x_{2t+1})
	\end{align*}
	for some quadratic form $T(x_1, \ldots, x_{2t+1}) \in \Q_p[x_1, \ldots, x_{2t+1}]$. By the inductive assumption, the quadratic form $T$ vanishes on a $t$-dimensional linear space $A$ over either $\Q_p$ itself or a suitable quadratic extension. Clearly, $A$ is orthogonal to the line spanned by $\B e_{2t+2} = (0, \ldots, 0,1) \in \Q_p^{2t+2}$, and we conclude that $Q$ vanishes on the $(t+1)$-dimensional linear space spanned by $A$ and $\B e_{2t+2}$. Hence the conclusion also holds true for $t+1$ in place of $t$. This finishes the proof.
\end{proof}

Since $n \ge 10$, we infer from \cite{L:52} that $F$ has a non-trivial $p$-adic zero $\B x_0$. In order to find a projective $\Q_p$-line on $\cal V$, using the notation from the introduction, we need to find a non-trivial solution $\B y \in \Q_p^n$ to the system \eqref{eq:sys-1} with the property that $\B x_0$ and $\B y$ are linearly independent.
Just like in our previous arguments, we choose a linear equation $M(\B y)=0$ not satisfied by $\B x_0$ in order to guarantee that our solution $\B y$ is distinct from $\B x_0$. It is therefore enough to solve the augmented system
\begin{align*}
    M(\B y)=\Phi(\B x_0, \B x_0, \B y)=\Phi(\B x_0, \B y, \B y)=F(\B y)=0.
\end{align*}
As $n \ge 22$, the two linear equations are satisfied on a $\Q_p$-linear space $V$ with affine dimension $\dim(V) = n-2 \ge 20$. By Lemma~\ref{l:bigquad} there exists a quadratic extension $\L$ of $\Q_p$, or possibly $\Q_p$ itself, such that $\Phi(\B x_0, \B y, \B y)$ vanishes on an $\L$-linear subspace $W \subseteq V$ of $\dim(W) \ge 10$. Hence by Lewis' classical result \cite[Theorem]{L:52} there exists $\B y \in W \setminus \{\bm 0\}$ such that $F(\B y)=0$. Then $\B x_0$ and $\B y$ span a line over $\L$, and by \cite[Lemma~2.2]{W:97linspcub}, this line can be pulled back to a line over $\Q_p$ itself. This finishes the proof of Theorem~\ref{t:Qp}\eqref{it:allp}.

We now turn to the proof of part \eqref{it:largep}. As before, by \cite{L:52} there exists $\B x_0 \in \Q_p^n \setminus \{\bm 0\}$ such that $F(\B x_0) = 0$, so we need to find a solution $\B y$ of the system of equations 
\begin{align*}
		\Phi(\B x_0, \B x_0, \B y) = \Phi(\B x_0, \B y, \B y) = F(\B y)= 0
\end{align*}
that is $\Q_p$-linearly independent of $\B x_0$.
The first equation $\Phi(\B x_0, \B x_0, \B y)=0$ defines a hyperplane $L$. As above, we also choose a hyperplane $M$ not containing $\B x_0$. It is then sufficient to solve the system 
\begin{align}\label{eq:f-q}
	f(\B y)=q(\B y)=0
\end{align}
which is obtained by restricting the forms $F(\B y)$ and $\Phi(\B x_0, \B y, \B y)$ to the hyperplane section $L \cap M$. Thus $f$ is a cubic and $q$ a quadratic form, both defined over $\Q_p$, in at least $14$ variables. We now choose an unramified quadratic extension $\L$ of $\Q_p$. As $p \ge 19$, the residue class field of $\L$ over $\Q_p$ has cardinality $p^2 > 293$, and the existence of a non-trivial solution $\B y$ of \eqref{eq:f-q} over $\L$ then follows from work of Zahid \cite[Theorem~1.1]{Zahid:11}. Thus again, $\B x_0$ and $\B y$ span a line over $\L$, which can be pulled back to $\Q_p$ by \cite[Lemma~2.2]{W:97linspcub}. This completes the proof of Theorem~\ref{t:Qp}\eqref{it:largep}.

\section{A smooth cubic hypersurface without rational lines}
Cubic hypersurfaces of projective dimension $9$ or lower do not necessarily contain $p$-adic lines for all primes $p$. In order to construct a counterexample, we build on recent work by Debarre, Laface and Roulleau \cite{DLR:17}. These authors investigated the existence or otherwise of lines on cubic hypersurfaces over finite fields. Among other things, they show that any cubic hypersurface of dimension $5$ or greater contains a line over $\F_q$ for all $q$ (Theorem 6.2 in their paper), and every smooth cubic fourfold over $\F_q$ contains a line when $q \not\in \{3,4\}$ (Theorem 5.2 for $q \ge 5$ and Corollary 5.4 for $q=2$). The cases when $q \in \{3, 4\}$ are not resolved, but the authors conjecture that all smooth cubic fourfolds over finite fields contain a line.

This leaves us to consider cubic threefolds. Here, we have a different picture: Smooth cubic threefolds over $\F_q$ contain lines whenever $q \ge 11$ (see \cite[Theorem~4.5]{DLR:17}), but while the cases $q \in \{7,8,9\}$ are open, over the fields $\F_2$, $\F_3$, $\F_4$ and $\F_5$ we now have explicit examples of non-singular cubic forms in $5$ variables that define hypersurfaces which do not contain any lines over $\F_q$.
\begin{lem}\label{l:counterex-5v}
    Let
    \begin{align*}
        H_2(x_1, \ldots, x_5) &=x_1^3+x_2^3+x_3^3+x_1^2 x_2 + x_2^2 x_3 + x_3^2 x_1 + x_1 x_2 x_3 + x_1 x_4^2 \\
        	& \quad  + x_1^2 x_4 + x_2 x_5^2 + x_2^2 x_5 + x_4^2 x_5,\\
        H_3(x_1, \ldots, x_5) &=2 x_1^3 + 2 x_2^3 + x_1 x_3^2 + x_2^2 x_4 + 2 x_3^2 x_4 + x_1^2 x_5\\
            &\quad +x_2 x_3 x_5 + 2 x_1 x_4 x_5 + 2 x_2 x_4 x_5 + 2x_4^2 x_5 + 2 x_4 x_5^2 + x_5^3,\\
        H_5(x_1, \ldots, x_5) &= x_1^3 + 2 x_2^3 + x_2^2 x_3 + 3 x_1 x_3^2 + x_1^2 x_4 + x_1 x_2 x_4 +  x_1 x_3 x_4\\
            &\quad + 3 x_2 x_3 x_4 + 4 x_3^2x_4 + x_2 x_4^2 + 4 x_3 x_4^2 + 3 x_2^2 x_5 +  x_1 x_3 x_5\\
            &\quad + 3 x_2 x_3 x_5 + 3 x_1 x_4 x_5 + 3 x_4^2 x_5 + x_2 x_5^2 + 3 x_5^3.
\end{align*}
    Then for $p \in \{2,3,5\}$, the hypersurface $H_p(x_1, \ldots, x_5)=0$ is smooth over $\C$ and contains no line over $\F_p$.
\end{lem}
\begin{pf}
    These examples are taken from \S 4.5.4 in \cite{DLR:17}. 
\end{pf}
We also need a supply of non-singular ternary forms that do not admit any non-trivial zero over $\F_p$.
\begin{lem}\label{l:counterex-3v}
    Let
    \begin{align*}
        G_2(x, y, z) &=x^3 + y^3 + z^3 + x y^2 + x z^2 + y z^2 + x y z,\\
        G_3(x, y, z) &=x^3 + y^3 + z^3 + 2 x y^2 + 2 x^2 z + 2 y z^2 + x y z,\\
        G_5(x, y, z) &=x^3 + y^3 + z^3 + x y^2 + x^2 z + y z^2.
    \end{align*}
    These polynomials are all non-singular over $\C$. Furthermore, for each $p \in \{2,3,5\}$, the equation $G_p(x,y,z)=0$ has only the trivial solution $x=y=z=0$ over $\F_p$.
\end{lem}
\begin{pf}
    This can be confirmed by a brute force calculation.
\end{pf}

We are now in a position to prove Theorem \ref{t:lower}. For $p \in \{2,3,5\}$, set
\begin{align} \label{eq:counterex-11v}
    F_p(x_1, \ldots, x_{11}) = H_p(x_1, \ldots, x_5) +   pG_p(x_6, x_7, x_8)+p^2G_p(x_9, x_{10}, x_{11}).
\end{align}
Suppose that the hypersurface given by $F_p=0$ contains a line $\ell$ over $\Q_p$.
Then, writing $\Z_p$ for the ring of $p$-adic integers, we can find two generators $\B x, \B y \in \Z_p^{11}$ of $\ell$, and it follows from Lemma 7 in \cite{Diet:08} that no generality is lost by assuming that $\B x$ and $\B y$ are $\F_p$-linearly independent. For any vector $\B a \in \Z_p^{11}$ write $\B a = (\B a_1, \B a_2, \B a_3)$, where $\B a_1 \in \Z_p^5$ and $\B a_2, \B a_3 \in \Z_p^3$. By \eqref{eq:counterex-11v}, as $F_p(\alpha \B x + \beta \B y)=0$ for all $\alpha, \beta \in \Z_p$, the cubic form $H_p$ vanishes over $\F_p$ on the line defined by the reductions of $\B x_1$ and $\B y_1$ modulo $p$. By Lemma \ref{l:counterex-5v}, this is only possible if $\B x_1$ and $\B y_1$ are linearly dependent modulo $p$. Hence there exists a linear combination $\B z = \lambda \B x + \mu \B y$ with $\lambda, \mu \in \Z_p$, not both divisible by $p$, such that all components of $\B z_1$ are divisible by $p$. Since $F_p(\alpha \B x + \beta \B y) = 0$ in $\Z_p$ for all $\alpha, \beta \in \Z_p$, in particular we must have $F_p(\B z)=0$. Our earlier conclusion that $\B z_1 \equiv \bm 0 \mmod p$ implies that $H_p(\B z_1) \equiv 0 \mmod{p^3}$, so reducing $F_p(\B z) \equiv 0 \mmod{p^2}$ forces $G_p(\B z_2)=0$ over $\F_p$. Hence by Lemma \ref{l:counterex-3v}, we discern that all components of $\B z_2$ are divisible by $p$, and thus $G_p(\B z_2)\equiv0 \mmod{p^3}$. In the same way, reducing modulo $p^3$ then shows that all components of $\B z_3$ are divisible by $p$. It follows that $\B z= \bm 0$ in $\F_p$, contradicting the fact that $\B x$ and $\B y$ are $\F_p$-linearly independent. Finally, the non-singularity of $F_p$ is inherited from the fact that both $H_p$ and $G_p$ are non-singular for all $p \in \{2,3,5\}$.

\appendix
\section{Mordell--Weil constructions}
We briefly outline a Mordell--Weil type theorem for smooth cubic hypersurfaces of high dimension. The secant-tangent construction for generating rational points on elliptic curves is well-known, and a classical result by Mordell \cite{Mordell} states that all rational points on an elliptic curve can be generated in this way from a finite set of generators. Motivated by this result, Manin asked \cite[p.~7]{Manin:cub} whether an analogous statement continues to hold for higher-dimensional cubic hypersurfaces. This question has very recently been answered in the affirmative by Papanikolopoulos and Siksek \cite{PapaSik}, provided that the hypersurface is of dimension at least 48. In their argument, they show that a point $\B x$ on the hypersurface generates another point $\B y$ if the two equations $F(\B y) = 0$ and $\B y \cdot \nabla F(\B x) = 0$ are satisfied. One of the main ingredients of their proof is therefore a result by the second author with Browning and Heath-Brown on the existence of rational points on the intersection of general cubic and quadratic equations \cite[Theorem~1.4]{BDHB:15}. However, if one applies Theorem~1.3 of the same paper instead, in a version akin to our Theorem~\ref{t:bdhb} above, one can obtain a stronger result.
\begin{thm}\label{t:PS}
    Let $\cal V \subseteq \PP^{n+1}$ be a smooth cubic hypersurface of dimension $n \ge 29$. Then there exists a point $A \in \cal V(\Q)$ such that $\Span A = \cal V(\Q)$.
\end{thm}
Observe that, just like in the work of Papanikolopoulos and Siksek whose methods our result is based on, the finite set generating $\cal V(\Q)$ is in fact a singleton. In the spirit of Manin's question, it would be interesting to know what the largest integer $n$ is such that there exists a smooth cubic hypersurface of dimension $n$ whose rational points do not lie in the span of a single point.

For the outline of the proof and all references, we refer to the work of Papanikolopoulos and Siksek \cite{PapaSik}.
In their work, the authors encounter the problem of showing that the variety $Y_{\B y, \B z} \subseteq \PP^{n+1}$ defined by the equations
\begin{align*}
    F(\B x) = 0, \qquad \Phi(\B x, \B x, \B z)=0, \qquad \Phi(\B x, \B y, \B y) = 0
\end{align*}
contains non-trivial rational points. The authors achieve this by invoking Theorem~1.4 of \cite{BDHB:15}. However, if the rank of the quadratic form can be controlled, a more efficient means to do this is to instead apply Theorem~1.3 of the same paper. We give an outline of the (very minor) changes in the argument that have to be effected in order to proceed in this manner.\smallskip

The following is a modification of Lemma~3.3 in \cite{PapaSik}.
\begin{lem}\label{l:PS}
    Suppose that $n \ge 29$. Let $\B y \in \cal V(\Q)$. Suppose that $\B z' \in \cal V(\R)$ is such that $Y_{\B y, \B z'}$ has a smooth real point $\B x'$. Then there is a non-empty open subset $U \subseteq \cal V(\R)$ such that when $\B z \in U \cap \cal V(\Q)$, then $Y_{\B y, \B z}(\Q) \neq \emptyset$.
\end{lem}
\begin{pf}
    It follows from Lemma~3.2 of \cite{PapaSik} that if the variety $Y_{\B y, \B z'}$ has a smooth real point $\B x'$, then for any open neighbourhood $V$ of $\B x'$ there exists an open neighbourhood $U_0$ of $\B z'$ such that for every $\B z \in U_0$ the variety $Y_{\B y, \B z}$ has a smooth real point $\B x \in V$. From Lemma~2 of \cite{HB:83} we see that the set of points $\B z \in \cal V$ having $\rk H_{\B z} \le n-1$ is a lower-dimensional submanifold of $\cal V$; we can therefore find an open subset $U \subseteq U_0$ such that for all $\B z \in U$ the variety $Y_{\B y, \B z}$ has a smooth real point $\B x \in V$, and furthermore the Hessian of $F$ at $\B z$ is of rank at least $n$. In addition, when $\B z \in \cal V(\Q) \cap U$, the variety $Y_{\B y, \B z}$ is defined over $\Q$.

    The variety $Y_{\B y,\B z}$ is defined by the intersection of a smooth cubic hypersurface, a quadric whose singular locus has dimension at most $0$, and a hyperplane. We restrict to the hyperplane and intersect with an additional auxiliary hyperplane that has been chosen generically. The resulting variety is the intersection of a smooth cubic hypersurface and a quadratic hypersurface whose singular locus has dimension at most $0$, defined over $\PP^{n-1}$. Furthermore, this variety has a non-singular real point. As noted in our proof of Theorem \ref{t:bdhb} above, the conclusion of Theorem~1.3 of \cite{BDHB:15} is still valid under these circumstances. We thus conclude that the set $Y_{\B y,\B z}$ contains rational points for all $\B z \in \cal V(\Q) \cap U$ whenever $n \ge 29$.
\end{pf}

The argument of \cite{PapaSik} can now be followed through literally, with our Lemma~\ref{l:PS} taking the role of their Lemma~3.3. This completes the proof of Theorem~\ref{t:PS}.

\bibliographystyle{amsplain}
\bibliography{fullrefs}
\end{document}